\documentclass[11pt]{amsart}

\usepackage{amsmath,amssymb,amscd,amsthm, amsfonts,verbatim,latexsym}
\usepackage{paralist}
\usepackage[mathscr]{eucal}

\usepackage[cmtip,arrow]{xy}
\usepackage{pb-diagram,pb-xy}
\usepackage{stmaryrd}

\usepackage{color}

\usepackage[pdftex]{hyperref}
\hypersetup{citecolor=blue,linktocpage}

\newtheorem{thm}{Theorem}[section]
\newtheorem{lem}[thm]{Lemma}

\newtheorem{prop}[thm]{Proposition}
\newtheorem{cor}[thm]{Corollary}

\newtheorem{assu-nota}[thm]{Assumption--Notation}

\theoremstyle{definition}

\newtheorem{rem}[thm]{Remark}
\newtheorem{ex}[thm]{Example}
\newtheorem{qst}[thm]{Question}

\newcommand{\inv}{^{-1}}
\newcommand{\up}[1]{^{(#1)}}

\newcommand{\C}{\mathbb C}
\newcommand{\Z}{\mathbb Z}

\newcommand{\pp}{\mathbb P}

\newcommand{\bG}{\mathbb G}
\newcommand{\bP}{\mathbb P}

\newcommand{\OO}{\mathcal O}
\newcommand{\nc}{\newcommand}
\nc{\cH}{{\mathcal H}}
\nc{\cA}{{\mathcal A}}
\nc{\cG}{{\mathcal G}}
\nc{\cC}{{\mathcal C}}
\nc{\cO}{{\mathcal O}}
\nc{\cI}{{\mathcal I}}
\nc{\cB}{{\mathcal B}}
\nc{\cY}{{\mathcal Y}}
\nc{\cK}{{\mathcal K}} 
\nc{\cX}{{\mathcal X}}
\nc{\cS}{{\mathcal S}}
\nc{\cE}{{\mathcal E}}
\nc{\cF}{{\mathcal F}}
\nc{\cZ}{{\mathcal Z}}
\nc{\cQ}{{\mathcal Q}}
\nc{\cN}{{\mathcal N}}
\nc{\cP}{{\mathcal P}}
\nc{\cL}{{\mathcal L}}
\nc{\cM}{{\mathcal M}}
\nc{\cT}{{\mathcal T}}
\nc{\cW}{{\mathcal W}}
\nc{\cU}{{\mathcal U}}
\nc{\cJ}{{\mathcal J}}
\nc{\cV}{{\mathcal V}}
\nc{\debar}{\bar\partial}

\DeclareMathOperator{\Div}{Div}
\DeclareMathOperator{\Pic}{Pic}
\DeclareMathOperator{\NS}{NS}

\DeclareMathOperator{\Alb}{Alb}
\DeclareMathOperator{\albdim}{albdim}
\DeclareMathOperator{\codim}{codim}
\DeclareMathOperator{\Hom}{Hom}
\DeclareMathOperator{\Ker}{Ker}

\DeclareMathOperator{\Spec}{Spec}
\DeclareMathOperator{\im}{Im}
\DeclareMathOperator{\rk}{rk}

\DeclareMathOperator{\pr}{pr}

\DeclareMathOperator{\Pf}{Pf}
\DeclareMathOperator{\Sing}{Sing}

\newcommand{\al}{\alpha}

\newcommand{\Si}{\Sigma}
\newcommand{\si}{\sigma}

\newcommand{\lra}{\longrightarrow}

\numberwithin{equation}{section}

\title[Continuous families of divisors\dots]{Continuous families of divisors,  paracanonical systems and a new inequality for varieties of maximal Albanese dimension }

\author{Margarida Mendes Lopes, Rita Pardini and Gian Pietro Pirola}

\thanks{
The first author is a member of the Center for Mathematical
Analysis, Geometry and Dynamical Systems (IST/UTL).  The second and the third author are members of G.N.S.A.G.A.--I.N.d.A.M.  This research was partially supported by FCT (Portugal) through program POCTI/FEDER and
Project PTDC/MAT/099275/2008 and by MIUR (Italy) through  PRIN 2008 ``Geometria delle variet\`a algebriche e dei loro spazi di moduli" and PRIN 2009 ``Spazi di moduli e Teoria di Lie"}
\begin{document}

\begin{abstract}
Given   a  smooth complex projective  variety $X$,  a line bundle $L$ of $X$ and   $v\in H^1(\OO_X)$,  we say that $v$ is {\em $k$-transversal} to $L$ if the complex $H^{k-1}(L)\overset {\cup v}{\to}H^k(L)\overset {\cup v}{\to} H^{k+1}(L)$ is exact. We prove that if $v$ is $1$-transversal to $L$ and $s\in H^0(L)$ satisfies $s\cup v=0$, then the first order deformation $(s_v, L_v)$ of the pair $(s,L)$ in the direction $v$ extends to an analytic  deformation.

We apply this result  to improve  known  results on the paracanonical system of a variety of maximal Albanese dimension, due to Beauville in the case of surfaces and to Lazarsfeld-Popa in higher dimension.
In particular, we prove  the inequality $p_g(X)\ge \chi(K_X)+q(X)-1$ for a variety $X$ of maximal Albanese dimension without irregular fibrations of Albanese general type. 

\noindent{\em 2000 Mathematics Subject Classification:} 14C20, 14J29, 32G10.
\end{abstract}
\maketitle
\setcounter{tocdepth}{2}
\tableofcontents

\section{Introduction}
The geometry of the divisors of an irregular algebraic variety is richer than in the regular case, due to  the existence of non trivial continuous families of divisors. Some of these families are intrinsically defined, as the paracanonical system introduced  below, which is  of fundamental importance in the study of irregular varieties of general type: indeed, the interest in the geometry  of the paracanonical system has been the motivation leading to the groundbreaking  results  of \cite{GL1} and \cite{GL2} on generic vanishing. 

Here we consider the following situation:  let $X$ be a smooth projective  variety,  let $\lambda$ be a class in the N\'eron-Severi group $\NS(X)$ and denote by $\Pic^{\lambda}(X)$ the connected component of $\Pic(X)$ that parametrizes isomorphism classes of line bundles with class $\lambda$. 
The effective divisors with class $\lambda$ are parametrized by a projective variety $\Div^{\lambda}(X)$ and the   characteristic map $c_{\lambda}\colon \Div^{\lambda}(X)\to \Pic^{\lambda}(X)$ is defined by $D\mapsto \OO_X(D)$. The fiber of $c_{\lambda}$ over a point  $L\in \Pic^{\lambda}(X)$ is naturally isomorphic to the linear system $|L|$. One  defines the {\em continuous rank} of $\lambda$ (this terminology is due to M.A. Barja) as $\rho(\lambda):= \min\{h^0(L)\ | \ L\in \Pic^{\lambda}(X)\}$:   if $\rho(\lambda)>0$ there is exactly one irreducible component   $\Div^{\lambda}(X)_{main}$ of $\Div^{\lambda}(X)$    that maps surjectively onto $\Pic^{\lambda}(X)$. The variety   $\Div^{\lambda}(X)_{main}$ is  called the {\em main component} of  $\Div^{\lambda}(X)$; its dimension is equal to  $\rho(\lambda)+q(X)-1$.

We study the problem of deciding whether the linear system $|L|$ is contained in $\Div^{\lambda}(X)_{main}$, and the closely related  question of deciding  whether $|L|$ is a component of $\Div^{\lambda}(X)$ (i.e., whether, in the terminology used  in \cite{beauville-annulation}, the system $|L|$ is {\em exorbitant}).

To this end, we introduce the notion of {\em $k$-transversality}  for a vector $v\in H^1(\OO_X)$:  we say that $v$ is $k$-transversal to $L$ if the complex given by cup product $$H^{k-1}(L)\overset {\cup v}{\to}H^k(L)\overset {\cup v}{\to} H^{k+1}(L)$$ is exact.  We prove the following:
\begin{thm}\label{thm:limit} Let $X$ be a smooth projective variety, let $0\ne L\in \Pic(X)$ and denote by $\lambda\in \NS(X)$ the class of $L$;  let $0\ne s\in H^0(L)$  and $v\in H^1(\OO_X)$ such that $s\cup v=0$, and denote by $D$ the divisor of zeros of $s$.  Then:
\begin{enumerate}
 \item  if $v$ is $1$-transversal to $L$, then $D$ belongs to the closure of $\Div^{\lambda}(X)\setminus |L|$;
 \item if $v$ is $k$-transversal to $L$ for every $k>0$ and $\chi(L)>0$, then $\rho(L)=\chi(L)$ and  $D\in \Div^{\lambda}(X)_{main}$. 
\end{enumerate}
\end{thm}

The idea of the proof is to construct inductively a formal deformation of the pair $(s,L)$ over $\Spec\C[[t]]$: the transversality assumption  on $v$ allows one to perform the inductive step (cf.  Proposition  \ref{prop:lift}). 
\medskip 

We apply Theorem \ref{thm:limit} to the analysis of the paracanonical system $\Div^{\kappa}(X)$ of a variety $X$  of maximal Albanese dimension, where $\kappa$ is the canonical class. One should observe that in this particular  case   the crucial deformation result Proposition  \ref{prop:lift} (and thus  Theorem \ref{thm:limit}) could  also be derived by the results in \cite[\S 3]{GL2} via Grothendieck duality; however, we believe   that the   statement   of Theorem  \ref{thm:limit},    that relates   the vanishing  of all  higher obstructions with  the $1$-transversality property, will prove useful for the analysis of other situations.
\smallskip

It is traditional to denote the paracanonical system  $\Div^{\kappa}(X)$ by $\cP_X$, or simply by  $\cP$.  The generic vanishing theorem of \cite{GL1} implies that the continuous rank $\rho(\kappa)$ of $\cP$ is equal to the Euler characteristic $\chi(K_X)$, hence for  $\chi(K_X)>0$ the paracanonical system  $\cP$ has a main component    $\cP_{main}$.
Paracanonical systems and conditions for $|K_X|$ to be or not exorbitant   have been studied in \cite{beauville-annulation} in the case of surfaces, and in \cite{bgg} in the general case.  Here we make some further progress in the understanding of  paracanonical systems. 
In dimension $>2$, we prove  the following result (see section 2 for the definitions of Albanese general type fibration and $V_k(X)$):
\begin{thm}\label{thm:para3} Let $n\ge 3$ be an integer and let $X$ be a smooth projective  $n$-dimensional variety   with irregularity $q\ge n+1$. If  $X$ has no Albanese general type fibration,  then $\chi(K_X)\ge q-n>0$ and:
\begin{enumerate}
\item $p_g\ge \chi(K_X)+q-1$;
\item $p_g(X)=\chi(K_X)+q -1$ if and only if $|K_X| \subset \cP_{main}$;
\item if $p_g(X)=\chi(K_X)+q-1$, then  $q=n+1$;
\item if,  in addition,  $0\in V_k(X)$ is an isolated point for $k>0$, then $p_g(X)=\chi(K_X)+q-1$ if and only if  $q=n+1$ and ${h(X)\choose n}\equiv 1 \mod 2$, where $h(X):=\sum_{j=0}^{\lfloor (n-1)/2\rfloor}h^{0,n-1-2j}(X)$.
\end{enumerate}
\end{thm}
Notice that  any smooth ample divisor $X$ in an abelian variety satisfies $p_g(X)=\chi(X)+q(X)-1$ (cf. Example \ref{ex:q=n+1}).  In section \ref{sec:examples} we give  examples that show that the assumption $q(X)\ge \dim X+1$ and the assumption that $X$ does not have fibrations of Albanese general type cannot be removed from Theorem \ref{thm:para3}.

To our knowledge, the inequality in Theorem \ref{thm:para3}  had not been conjectured  before, possibly due to the fact  for a surface $X$ one has  $p_g(X)=\chi(K_X)+q-1$.  Indeed,   \cite[Prop. 5.5]{bgg} gives necessary and sufficient numerical  conditions (that we use  in the proof of Theorem \ref{thm:para3}, (iv))  in order that $|K_X|$ be exorbitant,  assuming the converse inequality $p_g\le \chi+q-1$ (which  clearly holds when $|K_X|$ is not exorbitant). \medskip

Then we turn to the case of surfaces. In \cite{beauville-annulation} Beauville proved that if $X$ is a surface of general type of irregularity $q\ge 2$   without irrational pencils of genus $>q/2$ then the general canonical curve deforms to first order in  a $1$-transversal direction if and only if $q$ is odd.  This shows that $|K_X|$ is exorbitant for $q$ even and gives strong evidence for the fact that $|K_X|$ is not exorbitant if $q$ is odd. 
Here, thanks to Theorem \ref{thm:limit},   we are able to complete 
Beauville's results as follows: 
\begin{thm}\label{thm:para2}
Let $X$ be a surface  with $\chi(K_X)>0$ and irregularity $q\ge 2$ without irrational pencils of genus $>q/2$. Then: 
\begin{enumerate}
\item if $q$ is odd, then  $|K_X|\subset \cP_{main}$; 
\item if $q$ is even, then $\Sigma:=|K_X|\cap \cP_{main}=|K_X|\cap\overline{(\cP\setminus |K_X|)}$ is a  reduced and irreducible  hypersurface of degree $q/2$ of $|K_X|$. If, in addition, $X$ has no irrational pencil of genus $>1$, then  $\Sing(\Sigma)=\{[s]\in \Sigma \ |\ \rk(\cup s)<q-2\}$.
\end{enumerate}
\end{thm}

As an immediate  consequence  of Theorem \ref{thm:para2}, we are able to  answer in the case of surfaces the question, raised in \cite[\S 7]{BN}, of what is the  relation between the base schemes of $|K_X|$ and $\cP_{main}$. Since for $q\ge 3$ the variety $|K_X|\cap \cP_{main}$ is a non degenerate subvariety of $|K_X|$, we have:
\begin{cor}\label{cor:base2} Let $X$ be a surface of general type  that has no irrational pencil of genus $>q/2$. If $X$ has irregularity $q\ge 3$, then the base scheme of $\cP_{main}$ is contained in the base scheme of $|K_X|$.
\end{cor}

 Theorem \ref{thm:para3} and Theorem \ref{thm:para2} imply immediately  necessary conditions for the irreducibility of $\cP$:
\begin{cor} Let $X$ be a smooth projective $n$-dimensional variety with   irregularity $q\ge n+1$ that has no Albanese general type   fibration.  If the paracanonical system $\cP_X$ is irreducible, then $X$ is one of the following:
\begin{enumerate}
\item[(a)] a curve,
\item[(b)] a surface with  $q$  odd,
\item[(c)] a variety of dimension $n\ge 3$ with $q=n+1$ and $p_g(X)=\chi(K_X)+q-1$.
\end{enumerate}
\end{cor}

{\em Acknowledgments:} we wish to thank Mihnea Popa for some  useful mathematical communications.
\bigskip

{\bf Notation and conventions:} We work over the complex numbers; all varieties are projective.
If $X$ is a smooth projective variety,  we denote as usual  by  $p_g(X)$ the {\em geometric genus} $h^0(K_X)=h^n(\OO_X)$,  and by $q(X)$ the {\em irregularity} $h^0(\Omega^1_X)=h^1(\OO_X)$.   Recall that by Hodge theory  there is an antilinear isomorphism $H^0(\Omega^1_X)\to H^1(\OO_X)$ that we denote by  $\alpha\mapsto\bar \alpha$.
 We use the standard notations $\Pic(X)$, resp. $\NS(X)$,  for the group of divisors   modulo linear, resp. algebraic,  equivalence; given  $L\in \Pic(X)$, we denote by $\chi(L)$ its  Euler characteristic.\\
 For $V$  a complex vector space and $r\ge 1$ an integer, $\bG(r,V)$  denotes the Grassmannian of $r$-dimensional vector subspaces of $V$.

\section{Preliminaries}\label{sec:prelim}

We recall several known results on irregular varieties that are used repeatedly throughout the paper.
\subsection{Albanese dimension and irrational fibrations}\label{ssec:albanese}
Let $X$ be a smooth projective variety of dimension $n$. 
The {\em Albanese dimension} $\albdim(X)$ is defined as the dimension of the image of the Albanese map of $X$; in particular,  $X$ has {\em maximal Albanese dimension} if its Albanese map is generically finite onto its image and it is  of {\em Albanese general type} if in addition $q(X)>n$. For a normal  variety $Y$, we define the Albanese variety $\Alb(Y)$ and all the related notions  by considering  any smooth projective model of $Y$.

An {\em irregular fibration} $f\colon X\to Y$ is a morphism with positive dimensional connected fibers onto a normal  variety $Y$ with $\albdim Y=\dim Y>0$; the map $f$ is called an  {\em Albanese general type  fibration} if in addition $Y$ is of Albanese general type. If $\dim Y=1$, then $Y$ is a smooth curve of genus $b>0$; in that case, $f$ is called an {\em irrational pencil of genus $b$} and it is an Albanese general type fibration if and only if $b>1$. 

Notice that if $q(X)\ge n$ and $X$ has no Albanese general type fibration, then $X$ has  maximal Albanese dimension.

The so-called generalized Castelnuovo--de Franchis Theorem (see  \cite[Thm. 1.14]{catanese-irregular} and Ran \cite{ran}) shows  how  the existence of Albanese general type  fibrations  is detected by  the cohomology of $X$:
\begin{thm}[Catanese, Ran]\label{thm:cast-cat} 
The smooth projective variety $X$ has an Albanese general type  fibration $f\colon X\to Y$ with $\dim Y\le k$ if and only if there exist independent $1$-forms $\omega_0,\dots \omega_k \in H^0(\Omega^1_X)$ such that $\omega_0\wedge \omega_1\wedge \dots\wedge  \omega_k =0\in H^0(\Omega^{k+1}_X)$.
\end{thm}
A closely related result,  due to Green and Lazarsfeld, is recalled in the next section (Theorem \ref{thm:cast-gl}).

\subsection{Generic vanishing}

Let $X$ be a projective variety of dimension $n$ and let $L\in \Pic(X)$; 
the  $i$-th cohomological support locus of   $L$ is defined as  $V_i(L, X):=\{\eta\ | \ h^i(L+\eta)>0\}\subseteq \Pic^0(X)$, $i=0,\dots n$. The cohomological support loci are closed by the semi-continuity theorem.
One  says  that {\em generic vanishing holds for $L$} if $V_i(L,X)$ is a proper subvariety of $\Pic^0(X)$ for $i>0$.

We identify, as usual, the tangent space to $\Pic^0(X)$ at any point with $H^1(\OO_X)$, so we regard the elements of $H^1(\OO_X)$ as tangent directions. Given  $0\ne v\in H^1(\OO_X)$, the {\em derivative complex} of $L$ in the direction $v$ is:
\begin{equation}\label{eq:derivative}
0\to H^0(L)\overset{\cup v}{\lra} H^1(L)\overset{\cup v}{\lra}\dots \overset{\cup v}{\lra} H^n(L)\to 0.
\end{equation}
 We say that $v$ is {\em $k$-transversal} to $L$ if  the $k$-th cohomology group of the complex \eqref{eq:derivative} vanishes. This terminology is explained by the following key result (cf. \cite[\S 1]{GL1}, in particular Thm. 1.6):
\begin{prop}[Green-Lazarsfeld]\label{prop:derivative} Let $X$ be a smooth projective variety and let $L\in \Pic(X)$; consider a point $\eta\in V_k(L,X)$  and a nonzero direction $ v\in H^1(\OO_X)$. Then:
\begin{enumerate}
\item if $v$ is $k$-transversal to $L+\eta$, then $v$ is not tangent to $V_k(L,X)$ at $\eta$;
\item if $\eta\in V_k(L,X)$ is  general, then  either  $v$ is  $k$-transversal to $L+\eta$ or  both  maps in the sequence 
$$H^{k-1}(L+\eta)\overset{\cup v}{\lra} H^k(L+\eta)\overset{\cup v}{\lra} H^{k+1}(L+\eta)$$ vanish.
\end{enumerate}
\end{prop}

Proposition \ref{prop:derivative} has the following immediate consequence (in fact, the generic vanishing theorem  for $K_X$ was first proven in \cite{GL1} using this argument):
\begin{cor}\label{cor:gv}
Let $X$ be a smooth projective variety and let $L\in \Pic^0(X)$; if there exists $\eta\in \Pic^0(X)$ and $v\in H^1(\OO_X)$ such that $v$ is $k$-transversal to $L+\eta$ for every $k>0$, then generic vanishing holds for $L$.
\end{cor}

When $L=K_X$, we omit $L$ from the notation and simply write $V_i(X)$; the loci $V_i(X)$ are also  called {\em generic vanishing loci} of $X$. We recall the main facts about them: 
\begin{thm}[Green-Lazarsfeld, Simpson]\label{thm:gv-summary} Assume that $X$ has maximal Albanese dimension. Then 
\begin{enumerate}
\item $V_i(X)$ has codimension $\ge i$ in $\Pic^0(X)$ for $i=0,\dots n$;
\item  $V_0(X)\supseteq V_1(X)\supseteq \dots \supseteq V_n(X)=\{0\}$;
\item the components of $V_i(X)$ are translates of complex subtori of $\Pic^0(X)$ by torsion points.
\end{enumerate}
\end{thm}
\begin{proof} (i) is  \cite[Thm. 2.10]{GL1},  (ii) is \cite[Lem. 1.8]{theta}. The fact that the components of $V_i(X)$ are translates of subtori is the main result of \cite{GL2}. Finally, the fact that they are translates by torsion points is proven in \cite{simpson}.
\end{proof}

For  $L$  algebraically equivalent to $K_X$, the  condition that $v$ is $k$-transversal to $L$ is related to the existence of irregular fibrations, as shown by the following result (\cite[Thm. 5.3]{GL2}), that can also be regarded as a generalization of the Castelnuovo-de Franchis theorem:
\begin{thm}[Green-Lazarsfeld]\label{thm:cast-gl} Let $X$ be a smooth projective variety of maximal Albanese dimension, let  $\eta\in \Pic^0(X)$,  let $0\ne v\in H^1(\OO_X)$ and set $\omega:=\bar{v}\in H^0(\Omega^1_X)$. If $v$ is not $k$-transversal to $K_X+\eta$, then there exists an irregular fibration $f\colon X\to Y$ with $\dim Y\le n-k$ such that $\omega$ is the pull-back of  a  $1$-form $\alpha\in H^0(\Omega^1_{\Alb(Y)})$.
\end{thm}
Proposition \ref{prop:derivative} and  Theorem \ref{thm:cast-gl} for $\eta=0$, give the following:
\begin{cor}\label{cor:isolated}
If $X$ has no irregular fibration, then $0$ is an isolated point of $V_i(X)$ for every $i>0$.

\end{cor}

\section{Proof of Theorem \ref{thm:limit}}\label{sec:continuous}

The key step in the proof of Theorem \ref{thm:limit} is a deformation result for sections of $L$ that we state next. Denote by  $R$ the ring $\C[[t]]$ of formal power series; for $0\ne v\in H^1(\OO_X)$,   there is a morphism $\phi_v\colon \Spec R\to \Pic^{\lambda}(X)$ defined by $t\mapsto L\otimes \exp(tv)$, where $\exp\colon H^1(\OO_X)\to \Pic^0(X)$ denotes the exponential mapping.
 Pulling back via $\phi_v$  a Poincar\'e line bundle on $X\times \Pic(X)$, one obtains a deformation $L_v(t)$ of $L$ over $\Spec R$ that we call the {\em (formal) straight line deformation of $L$ in the direction $v$}. 
   
  \begin{prop}\label{prop:lift}
   Let $X$ be a smooth projective variety, let   $L\in \Pic(X)$ be  a line bundle and  let    $s\in H^0(L)$ be a nonzero section. If $v\in H^1(\OO_X)$ is $1$-transversal  to $L$ and  $s\cup v=0$, 
 then  the formal  straight line deformation $L_v(t)$ of $L$ in the direction $v$  can be lifted to a deformation $(s_v(t), L_v(t))$ of the pair $(s,L)$.
 \end{prop}
 
 Granting  Proposition \ref{prop:lift} for the moment, we prove Theorem \ref{thm:limit}.
 \begin{proof}[Proof of Thm. \ref{thm:limit}]
 (i) By Proposition \ref{prop:lift},  there exists a formal deformation $(s_v(t), L_v(t))$ of the pair $(s,L)$  over $\Spec R$, where  $L_v(t)$ is the straight line deformation of $L$ in the direction $v$. By Artin's convergence Theorem (\cite[Thm. 1.2]{artin}), there exists an analytic deformation $(\tilde{s}(t), \tilde{L}(t))$ of $(s,L)$  over a small disk around $0\in \C$ such that $(s_v(t), L_v(t))=(\tilde{s}(t), \tilde{L}(t)) \mod t^2$. 
 \smallskip
 
 (ii) Since $v$ is $k$-transversal, $L$ satisfies generic vanishing by Proposition \ref{prop:derivative}, hence $\rho(L)=\chi(L)>0$. 
  To prove the claim, it  is enough to show that for $t\ne 0$ small  the analytic deformation $(\tilde{s}(t), \tilde{L}(t))$ is not contained in an irreducible  component $Z$ of $\Div^{\lambda}(X)$ different from the main component. By the semicontinuity of the dimension of cohomology groups, it follows that the image in $\Pic^{\lambda}(X)$ of such a component is contained in $L+V_k(L,X)$ for some $k>0$. By Proposition \ref{prop:derivative}, the vector $v$, being $k$-transversal to $L$, is not tangent to $L+V_k(L,X)$ at the point $L\in \Pic^{\lambda}(X)$, hence $\tilde{L}(t)\notin L+V_k(L,X)$ for $t\ne 0$ small. 
 \end{proof}
 
 The rest of the section is devoted to the proof of Proposition  \ref{prop:lift}.  We need to state and prove some  technical results. 
  Let $D$ be the  divisor of zeros of $s$;   for $n\ge 1$ consider the twisted restriction sequence
 \begin{equation}\label{eq:res} 0\to \OO_X((n-1)D)\to \OO_X(nD)\to \OO_D(nD)\to 0\end{equation}
 and denote by  $d_n\colon H^0(\OO_D(nD))\to H^1(\OO_X((n-1)D))$ the induced maps in cohomology.

 We  use Dolbeault cohomology.
 We consider the
sheaf $\cC^\infty(L)$ of $C^{\infty}$ sections of the line bundle $L$ and we set  $\cC^\infty(L)_D: =\cC^\infty(L)/\cC^\infty$; in addition, we denote by $\cA^{0,q}$, $\cA^{0,q}(L)$ the sheaves  of $C^{\infty}$ 
$(0,q)$-forms, of $C^{\infty}$  $(0,q)$-forms  with values in $L$.
We have a commutative  diagram 

\begin{gather*}
    \begin{diagram}
 \node{\cO_{S}} \arrow{e,t}{s}
      \arrow{s,r}{}
      \node {L}
\arrow{s,r}{} 
      \arrow{e}{} \node{L|_D} \arrow{s,r}{} 
 \\
 \node{\cC^\infty} \arrow{e,t}{}
      \arrow{s,r}{\over \partial}
      \node {\cC^\infty(L)} \arrow{s,r}{\over \partial} 
      \arrow{e}{} \node{\cC^\infty(L)_D}
 \\
 \node{\cA^{0,1}} \arrow {e,t}{}
      \node{\cA^{0,1}(L)} {}
  \end{diagram}
  \end{gather*}
 \begin{lem}
 \label{lem:dolbeault} Let $k\ge 1$ be an integer, let $\si\in H^0(L|_D)$ and set $v:=d_1\si$, let $\psi\in \cC^{\infty}(L)$ be a lift of $\si$ and let $\zeta\in \cC^{\infty}(L)$. 
 
  If $v$ is $1$-transversal to $L$ and   $\psi^k\debar \zeta\in \cA^{0,1}((k+1)L)$ is $\debar$-closed, 
 then $\psi^k\debar \zeta$ is $\debar$-exact.
 \end{lem} 
 \begin{proof}
 Since $0=\debar(\psi^k\debar \zeta)=k\psi^{k-1}\debar\psi\wedge \debar \zeta$, we have $\debar \psi\wedge \debar \zeta=0$ on the support of $\psi$. On the other hand, $\debar \psi=0$ outside the support of $\psi$, hence $\debar \psi\wedge \debar \zeta=0$ on all of $X$ and therefore $\zeta\debar \psi\in \cA^{0,1}(L)$ is closed. 
Since the class $v\in H^1(\OO_X)$ is represented  by $\debar \psi\in \cA^{0,1}$ and  $\debar \psi\wedge \zeta\debar \psi=0$, we have  $v\cup [\zeta\debar \psi]=0$. By   
  the assumption that $v$ is $1$-transversal to $L$,  there exists $g\in H^0(L)$ and $\alpha\in \cC^{\infty}(L)$ such that $$\zeta\debar \psi=g\debar \psi+\debar \alpha=\debar(g\psi+\alpha).$$
 Then we have 
 $$\psi^k\debar \zeta=\debar(\psi^k\zeta)-k\psi^{k-1}\zeta\debar \psi=\debar(\psi^k\zeta)-k\psi^{k-1}\debar(g\psi+\alpha).$$
   Since $k\psi^{k-1}\debar(g\psi)=\debar(\psi^kg)$, we get:
\begin{equation}\label{eq:bordo}
\psi^k\debar \zeta=\debar(\psi^k(\zeta-g))-k\psi^{k-1}\debar \alpha.
\end{equation}
For $k=1$, \eqref{eq:bordo} is precisely the claim,   and for $k>1$ it  gives the inductive step.
 \end{proof}
  
 \begin{lem} \label{lem:obstruction}
  Let $\si\in H^0(\OO_D(D))$ and set $v:=d_1\si$.  Then: 
  \begin{enumerate}
  \item $v\cup d_2\si^2=0$,
  \item if  $\xi\in H^0(L)$ is such  that $d_2\si^2=v\cup 2 \xi$, then $d_2(\si-r_1(\xi))^2=0$, where $r_1\colon H^0(L)\to H^0(\OO_D(D))$ is the restriction map,
  \item if $d_2\si^2=0$ and $v$ is $1$-transversal to $L$, then $d_n\si^n=0$ for every $n\ge 2$.
  \end{enumerate}
 \end{lem}
 \begin{proof}
(i)  We regard $\si$ as a global section of  $\cC^\infty(L)_D$ and we lift it to a global section
$\psi\in \cC^\infty(L)$. Then ${\bar\partial} \psi $ represents  the  class  $v=d_1\si \in H^1(\cO_X)$ and
 $\bar{\partial} \psi^2=2\psi{\bar\partial} \psi $ represents the class $d_2\si^2$.
But then (i) follows,  since $$v\cup d_2\si^2=[ 2\psi\bar\partial {\psi}\wedge\bar\partial\psi]=0.$$

(ii) The function $\psi-\xi$ is a $\cC^{\infty}$ lift of $\si-r_1(\xi)$, hence $d_2(\si-r_1(\xi))^2$ is represented by $2(\psi-\xi)\debar \psi$. Since, as shown in the proof of (i),  $\debar\psi$ represents $v$ and $2\psi\debar \psi$ represents $d_2\si^2$, we have  $[2(\psi-\xi)\debar \psi]=d_2\si^2 -v\cup 2\xi=0$. 
\smallskip

(iii) Since $d_2\si^2=0$, there exists  $\zeta\in \cC^{\infty}(L)$  such that $\psi\debar \psi=\debar \zeta$.
For  $n\ge 2$ the class $d_n\si^n$ is represented by $\debar(\psi^n)=n\psi^{n-1}\debar \psi = n\psi^{n-2}\debar \zeta$, hence it vanishes   by Lemma \ref{lem:dolbeault} since $v$ is $1$-transversal to $L$.
 \end{proof}
\begin{proof}[Proof of Prop.  \ref{prop:lift}] Denote by $D$ the divisor of zeros of $s$ and  for $n\ge 1$  let $r_n \colon H^0(\OO_X(nD))\to H^0(\OO_D(nD))$ be the restriction map. 
Let $\{U_i\}$ be an   open cover of $X$ consisting of polydiscs, denote by $g_{ij}\in \OO_X^*(U_{ij})$  a system of transition functions for $L=\OO_X(D)$, by $s_i\in \OO_X(U_i)$ local representations of $s$  and by $a_{ij}\in \OO_X(U_{ij})$ a cocycle that represents $v$.  Set $V_i:=U_i\times \Spec R$;   transition functions for $L_v(t)$ with respect to the open cover $\{V_i\}$ of $X\times \Spec R$ are given by : 
  \begin{equation}
g_{ij}(t)=g_{ij}\exp(ta_{ij})=g_{ij}(1+a_{ij}t+\frac{a_{ij}^2t^2}{2}+\dots +\frac{a_{ij}^nt^n}{n!}+\dots ), 
 \end{equation}
We  look for  functions  $s_i(t)\in  \OO_{X\times \Spec R}(V_i)$ that satisfy:
\begin{equation}\label{eq:section} 
s_i(t)=g_{ij}(t) s_j(t).
\end{equation}
We write formally  $s_i(t)=s_i\exp(\si_it+\sum_{r\ge 2}\tau_i^{(r)}t^r)$, with $\si_i$,  $\tau_i^{(r)}$ meromorphic functions on $U_i$ with poles only on $D$,   and we start by solving \eqref{eq:section} $\mod t^2$, namely we look for functions  $\si_i\in \OO_X(D)(U_i)$ that on $U_{ij}$ satisfy 
\begin{equation}\label{eq:sec1}
\si_i=a_{ij}+\si_j.
\end{equation}
Equivalently, we look for a section $\si\in H^0(\OO_D(D))$ such that $d_1\si=v$.
Consider the long cohomology sequence associated with the standard restriction sequence on for $D$:
$$\dots \to H^0(\OO_X(D))\overset{r_1}{\lra} H^0(\OO_D(D))\overset{d_1}{\lra}  H^1(\OO_X)\overset{\cup s}{\lra} H^1(\OO_X(D))\to \dots $$
Since $s\cup v=0$ by assumption,  equation \eqref{eq:sec1} can be solved, and any two solutions differ by a section of $H^0(\OO_X(D))$.
The second step is to solve for every $i,j$:
\begin{equation}\label{eq:sec2}
 \tau\up2_i=\tau\up2_j \ \mbox{on} \ U_{ij}, \quad \tau_i\up2+\frac{\si_i^2}{2} \in \OO_X(D)(U_i).
\end{equation} 
Namely, we look for a section $\tau\up2\in H^0(\OO_X(2D))$ that lifts the section  $-\frac{\si^2}{2}\in H^0(\OO_D(2D))$. Hence we must have  $d_2\si^2=0$. By Lemma \ref{lem:obstruction} (i), $v\cup d_2\si^2=0$, hence by the assumption that $v$ is transversal to $L$ there exists $\xi\in H^0(\OO_X(D))$ such that $d_2\si^2=v\cup 2\xi$. 
Hence we have $d_1(\si-r_1(\xi))=d_1\si=v$ and by Lemma \ref{lem:obstruction} (ii)  we have  $d_2(\si-r_1(\xi))^2=0$.

Therefore we may replace $\si$ by $\si-r_1(\xi)$ and then  solve \eqref{eq:sec2}.
The third step consists in solving for every $i,j$:
\begin{equation}\label{eq:sec3}
 \tau\up3_i=\tau\up3_j \ \mbox{on} \ U_{ij}, \quad \tau_i\up3+\si_i \tau_i\up2+\frac{\si_i^3}{3} \in \OO_X(D)(U_i).
\end{equation} 
As in the previous step, \eqref{eq:sec3} has a solution $\tau\up3\in H^0(\OO_X(3D))$ iff:
\begin{equation}\label{eq:ob3}
d_3(3\si r_2(\tau\up2) +
\si^3)=0.
\end{equation} 
 Since by construction $r_2(\tau\up2)$ is a nonzero multiple of $\si^2$, this is equivalent to $d_3\si^3=0$. Hence \eqref{eq:sec3} can be solved by Lemma \ref{lem:obstruction} (iii).
By the same argument, one shows inductively the existence of $\tau\up{r}$ for every $r\ge 4$.
\end{proof}

\section{Paracanonical systems}\label{sec:paracanonical}
\subsection{Proof of Theorem \ref{thm:para3}}

We let $X$ be a smooth projective variety with $\albdim X=\dim X=n\ge 2$; for the sake of brevity,  we write $p_g, q, \chi$ instead of $p_g(X), q(X), \chi(K_X)$. We assume that $\chi>0$, so that we can consider  the main paracanonical system.

We write $\pp:=\pp(H^1(\OO_X))$ and  define an incidence subvariety $\cI\subset \pp\times |K_X|$ as follows:
$$\cI:=\{([v],[s])\ | \ v\cup s=0\}.$$
Let  $\cI_{main}\subseteq \cI$ be the closure of the open subset consisting of pairs $([v],[s])$ such that $v$ is $1$-transversal to $K_X$.   We consider  the two projections of $\pp\times |K_X|$ and we set $\Si\subseteq |K_X|$ and respectively $\Si_{main}\subseteq |K_X|$,  the image of $\cI$, respectively $\cI_{main}$, via the second projection. Hence $\Si$ is the locus  of canonical divisors that deform to first order in some non zero  direction $v$, and $\Si_{main}$ is the closure of the locus of canonical divisors that deform to first order in some direction $v$ that is $1$-transversal to $K_X$.
The key observation is the following:
\begin{lem}\label{lem:key} One has inclusions:
$$\Sigma_{main}\subseteq \cP_{main}\cap |K_X|\subseteq |K_X|\cap \overline{(\cP\setminus |K_X|)}\subseteq \Sigma.$$
\end{lem}
\begin{proof}
By Theorem \ref{thm:cast-gl}, a vector $v\in H^1(\OO_X)$ is $1$-transversal to $K_X$ if and only if it is $k$-transversal to $K_X$ for every $k>0$. 
Hence the first inclusion is a consequence of Theorem \ref{thm:limit}. The second one is obvious. 

To prove the last one, consider   $D=(s)\notin \Sigma$; the sequence $$0\to H^0(\OO_S)\to H^0(\OO_S(D))\to H^0(\OO_D(D))\to 0$$ is exact. Since $H^0(\OO_D(D))$ is the tangent space to $\cP$ at the point $D$, this means that $\cP$ and 
 $|K_X|$ have the same tangent space at $D$, hence $\cP$ is smooth at  $D$, $|K_X|$ is exorbitant and $D\notin \cP_{main}$.  \end{proof}
Now  we start to study the geometry of the varieties that we have introduced:
 \begin{lem}\label{lem:count}
\begin{enumerate}
\item $\cI_{main}$ is irreducible of dimension $\chi+q-2$;
\item if $0\in V_k(X)$ is an isolated point for $k>0$, then $ |K_X|\cap \overline{(\cP\setminus |K_X|)}=\Sigma_{main}$.
\end{enumerate}
\end{lem}
\begin{proof}
Denote by $\pr_i$, $i=1,2$, the restrictions to $\cI$, $\cI_{main}$ of the two projections of $\pp\times |K_X|$.

(i) By Theorem \ref{thm:cast-gl}, a vector $v\in H^1(\OO_X)$ is $1$-transversal to $K_X$ if and only if it is $k$-transversal to $K_X$ for every $k>0$. Hence, if  $v\in H^1(\OO_X)$ is $1$-transversal to $K_X$, then $\pr_1\inv ([v])$ is a projective space of dimension $\chi-1$. Since by Theorem \ref{thm:cast-gl}, the classes $[v]$ such that  $v$ is $1$-transversal to $K_X$ form a nonempty open subset of $\pp(H^1(\OO_X))$, the claim follows immediately.
\smallskip

(ii) By Proposition \ref{prop:derivative},  the point $0\in V_k(X)$ is  isolated for $k>0$ if and only if every $0\ne v\in H^1(\OO_X)$ is $k$-transversal to $K_X$ for $k>0$, hence, as explained  in the proof of (i),  if and only if every $0\ne v\in H^1(\OO_X)$ is $1$-transversal to $K_X$. Therefore   we have $\cI=\cI_{main}$, and the claim follows by Lemma \ref{lem:key}.
\end{proof}

Next we  need  a linear algebra result, which is an infinitesimal version of   the so-called ker/coker lemma:
\begin{lem}\label{lem:kercoker}
Let $V,W$ be complex vector spaces, let $X\subset \Hom(V,W)$ be an irreducible closed subset  and let $f\in X$ be general.

  If $g$ is tangent to $X$ at the point $f$,  then   $g(\ker(f))\subseteq \im f$. \end{lem}
 \begin{proof}
 By the generality of $f$, we may assume that $X$ is smooth at $f$ and that $\rk f$ is equal to the maximum of $\rk h$ for $h\in X$; we set $r:=\dim \Ker f$. 
 Then there is an analytic map  $F\colon \Delta \to X$, where $\Delta$ is a small disk around $0\in \C$, such that $F(t)=f+ gt+\dots $; then $t\to U(t):=\ker F(t)$ defines an analytic map $\Delta\to \bG(r, V)$, where $\bG(r,V)$ is the Grassmannian. Given $u\in \ker f$, we can lift $t\to U(t)$ to the universal family on $\bG(r,V)$, namely there exists  an analytic map  $u\colon \Delta\to V$  such that $u(t)=u+u_1t+\dots$ and $F(t)u(t)=0$. Hence $0=(F(t)u(t))'|_{t=0}=fu_1+gu$, i.e.,  $gu=-fu_1\in \im f$.
 \end{proof}
 
 We can now complete the proof of Thm. \ref{thm:para3}:
\begin{proof}[Proof of Thm. \ref{thm:para3}]
Since $X$ has no Albanese general type    fibration and $q\ge n+1$, the  inequality $\chi(K_X)\ge q(X)-n$ holds by \cite{PareschiPopa}.
\medskip

(i) Consider the projection $\pr_2\colon \cI_{main}\to \Sigma_{main}\subseteq |K_X|$.   If for general $[s]\in \Sigma_{main}$ the fiber $\pr_2\inv ([s])$ is a point,
or, equivalently, if  the dimension $t$ of the kernel $N$ of the map $\cup s\colon H^1(\OO_X)\to H^1(K_X)$ is equal to 1, then  the inequality $\chi+q-1\le p_g$ follows by  Lemma \ref{lem:count}.  

So assume that  $t\ge 2$ and denote by $\pp(T)\subseteq |K_X|$ the projective tangent space to $\Sigma_{main}$ at a general  point $[s]$. Let $\al,\beta\in H^0(\Omega^1_X)$ be independent $1$-forms such that $\bar \alpha,  \bar \beta\in N$.
Given  $s'\in T\subset  H^0(K_X)$,  by Lemma \ref{lem:kercoker} there exists $\gamma \in H^0(\Omega^1_X)$ such that $s'\wedge \bar \alpha=s\wedge \bar \gamma$. Thus we have $s'\wedge \bar\alpha\wedge \bar\beta=s\wedge \bar \gamma \wedge \bar\beta=-s\wedge  \bar\beta\wedge \bar \gamma=0$, for every $\bar \alpha, \bar \beta\in N$. It follows that $T$ is annihilated by the image $S$ of the map $\wedge^2 N\otimes H^{n-2}(\OO_X)\to H^n(\OO_X)$. Denote by $\bG_N$ the subset of the Grassmannian $\bG(n, H^1(\OO_X))$ consisting of the $n$-dimensional vector subspaces  of $H^1(\OO_X)$ that intersect $N$ in dimension $\ge 2$. Since by assumption $X$ has no Albanese general type   fibration, by Theorem \ref{thm:cast-cat} the wedge product induces a finite map $\bG_N\to \pp(S)$.
Hence, we have:
$\codim \Sigma_{main}=\codim T\ge\dim S\ge  \dim \bG_N+1= (n-2)(q-n)+2(t-2)+1$ if $q\ge n+t -1$ and $\codim \Sigma_{main}\ge  n(q-n)+1$ if $q<n+t-1$.  Since $\dim\Sigma_{main}=\dim \cI_{main}-(t-1)=\chi+q-t$, this can be rewritten  as:
\begin{itemize}
\item[(a)]   $p_g-(\chi+q-t)\ge (n-2)(q-n)+2t-3$,  if $q\ge n+t-1$;
\item[(b)] $p_g-(\chi+q-t)\ge n(q-n)+1$,  if $q<n+t-1$.
\end{itemize}
In case (a) we obtain $p_g-(\chi+q-1)\ge (n-2)(q-n)>0$; in case (b), since $t\le q$,  we get $p_g-(\chi+q-1)\ge (n-1)(q-n)-n+ 2\ge (n-1)-n+2=1$.
\medskip

(ii) If $|K_X|\subset  \cP_{main}$, then $p_g\le \chi+q-1$ and therefore $p_g=\chi+q-1$ by (i). Conversely,   the proof of (i) shows that if $p_g=\chi+q-1$ we have $t=1$. Hence in this case $\dim\Sigma_{main}=\chi+q-2$ and therefore $\Sigma_{main}=|K_X|$ and 
 $|K_X|\subset \cP_{main}$ is not exorbitant. 
\medskip

(iii) By (ii), if $p_g=\chi+q-1$, then for general $s\in H^0(K_X)$   the kernel of $\cup s\colon H^1(\OO_X)\to H^1(K_X)$  is $1$-dimensional.
Let $s\in H^0(K_X)$ be general and take $0\ne \alpha\in H^0(\Omega^1_X)$   such that $s\wedge\bar\alpha=0$. 
We denote by $W$ the image of the map $H^0(K_X)\to H^1(K_X)$ defined by $w\mapsto w\wedge \bar \alpha$. By Lemma \ref{lem:kercoker}, for every  $w\in H^0(K_X)$  there exists $\beta \in H^0(\Omega^1_X)$ such that $w\wedge \bar \alpha =s\wedge \bar \beta$,  namely $W$ is contained in the  image of the map $\cup s\colon H^1(\OO_X)\to H^1(K_X)$. It follows that $W$ has dimension $\le q-1$. 

 Let 
$\bar{M}\subset H^1(\cO_X)$ be a subspace such that $<\bar{\alpha}>\oplus \bar{M}=H^1(\cO_X)$,  and consider 
$M\subset H^0(X,\Omega^1_X)$. Now take a  decomposable form 
$s_1=\beta_1\wedge...\wedge \beta_{n-1}\wedge \alpha$ with $\beta_i\in M$ independent. By the assumption that $X$ has no Albanese general type   fibration, we have $s_1\ne 0$,  hence $s_1\wedge \overline{s_1}\ne0$, and, a fortiori,    $s_1\wedge \bar {\alpha}\neq 0$. 
Hence the map $\bigwedge^{n-1}M\to W$ defined by $s\mapsto s\wedge\alpha\wedge \bar \alpha$ induces a finite morphism 
$\bG(n-1,M)\to \bP(W)$.  It follows that
$q-2\geq (n-1)(q-1-(n-1))=(n-1)(q-n)$, that is $(n-2)q\leq (n-2)(n+1)$  and thus $q\leq n+1$.

(iv) 
By (ii) the equality $p_g=\chi+q-1$ holds  if and only if $|K_X|\subset \cP_{main}$,   if and only if, since $0\in V_k(X)$ is isolated for $k>0$, $|K_X|$ is not exorbitant. In addition, by (iii) if $p_g(X)=\chi(X)+q-1$ then $q=n+1$.

 Hence, by \cite[Prop. 5.5]{bgg}, in our situation $|K_X|$ is not exorbitant if and only if the coefficient  $s_n$ of $t^n$ in the formal power series expansion in $\Z[[t]]$  of  the rational function $\Pi_{j=1}^n(1+jt)^{(-1)^{j+1}h^{0,n-j}}$ does not vanish. 
 
  As explained in \cite{bgg}, $s_n$ is the degree in $\pp(H^1(\OO_X))$ of a general fiber of the map $\cI_{main}\to \pp(H^0(K_X))$; since by the proof of (i) this fiber is either empty or a point, it is enough to compute $s_n$ modulo $2$. Finally, it is an elementary computation to show that $s_n\equiv {h(X)\choose n}\mod 2$.\end{proof}

\subsection{Proof of Theorem \ref{thm:para2}}  Let  $X$ be a surface  with irregularity $q\ge 2$ and  with  $\chi:=\chi(K_X)>0$ that has no irrational pencil of genus $>q/2$. Recall that  $p_g=\chi(K_X)+q-1$.

As explained in the introduction, part of the results that follow were already  proven in \cite{beauville-annulation}, but  in order to give a clear presentation of our results we prefer to give all the proofs. 
\begin{lem}\label{lem:dimension}
The irreducible components of $\cI$ distinct from $\cI_{main}$ have dimension $<p_g-1$.
\end{lem}
\begin{proof} Let $v\in H^1(\OO_X)$ and let $\alpha=\bar v\in H^0(\Omega^1_X)$: by Hodge theory  $v$ is not $1$-transversal to $K_X$ if and only if the  sequence   $H^0(\OO_X)\overset{\wedge \alpha}{\to} H^0(\Omega^1_X)\overset{\wedge \alpha}{\to}  H^0(K_X)$ is not exact. 
Hence by  the classical Castelnuovo-de Franchis Theorem (cf. also Theorem  \ref{thm:cast-cat} and \ref{thm:cast-gl}), the set of classes $[v]\in \pp:=\pp(H^1(\OO_X)$ such that $v$ is not  $1$-transversal to $|K_X|$ is the union of the (finitely many) mutually disjoint  linear subspaces $\pp(\overline{f^*H^0(K_B)})$,  where $f\colon X\to  B$ is a pencil of genus $b>1$. 
Fix such a pencil $f$ and write $W:= f^*H^0(K_B)$.  Let $v=\bar{\alpha}$, with $\alpha \in W$, and  choose a subspace $N\subset H^0(\Omega^1_X)$ such that $H^0(\Omega^1_X)=W\oplus N$:   again by the Castelnuovo-de Franchis theorem, we have $\alpha\wedge \beta\ne 0$ for every $0\ne \beta \in N$. Therefore $\int_X\alpha\wedge \beta\wedge\bar\alpha\wedge \bar \beta\ne 0$, hence  $\alpha\wedge \beta\wedge\bar\alpha\ne 0\in H^1(K_X)$. This shows that
 the linear map $\cup v \colon H^0(K_X)\to H^1(K_X)$ has rank $\ge q-b$. Hence $\pr_1\inv([v])$ has dimension $\le p_g-q+b-1$ and the preimage in $\cI$  of $\pp(\overline{f^*H^0(K_B)})$ has dimension $\le p_g-q+2b-2<p_g-1$, by the assumption that $b\le q/2$.

Assume that $Z$ is an irreducible component of $\cI$ distinct from $\cI_{main}$ and consider the first projection $\pr_1\colon \cI\to \pp$: by the definition  of $\cI_{main}$,   $\pr_1(Z)$ is contained in $\pp(\overline{f^*H^0(K_B)})$ for some irrational pencil $f\colon X\to B$, and the claim follows by the previous discussion.
\end{proof}

Given a $2$-form $s\in H^0(K_X)$, by Serre duality the linear map $\cup s\colon H^1(\OO_X)\to H^1(K_X)\cong H^1(\OO_X)^{\vee}$ induces a skew-symmetric bilinear  form on $H^1(\OO_X)$ that we denote by $c_s$. Clearly, $\Sigma$ is the subset of classes $[s]$ such that  $c_s$ is degenerate. 

\begin{lem}\label{lem:qeven-odd}
 \begin{enumerate}
\item If $q$ is odd, then $\Sigma=\Sigma_{main}=|K_X|$
\item If $q$ is even, then $\Sigma=\Sigma_{main}$ is the zero set of a polynomial $\Pf$ of degree $q/2$.
\end{enumerate}
\end{lem}
\begin{proof} 
(i) If $q$ is odd, then $c_s$, being skew-symmetric,  is degenerate for every $s\in H^0(K_X)$, hence $\Si=|K_X|$. By Lemma \ref{lem:dimension}, $\cI_{main}$ is the only component of $\cI$ that can dominate $|K_X|$, hence we also have $\Sigma_{main}=|K_X|$.
\smallskip

(ii) For $q$ even,  $\Sigma$ is the  zero locus of the pull back $\Pf$  of the Pfaffian polynomial  in $\bigwedge^2 H^1(\OO_X)^{\vee}$, hence  it is either a divisor or it is equal to $|K_X|$. The latter possibility cannot occur by dimension reasons, since for every $[s]\in\Sigma$ the fiber $\pr_2\inv([s])$ is an odd dimensional linear space. So $\Sigma$ is a divisor. Let $\Delta\subseteq\Sigma$ be an irreducible component: then $\pr_2\inv(\Delta)$ has a component of dimension $\ge p_g-1$, hence by Lemma \ref{lem:dimension} $\pr_2\inv(\Delta)$ contains $\cI_{main}$. It follows that $\Delta=\Sigma_{main}=\Sigma$.
\end{proof}

\begin{lem}\label{lem:sing} 
Assume that $q\ge 4$ is even and let  $[s]\in \Sigma$  be such that:
\begin{enumerate}[(a)]
\item the map $\cup s\colon H^1(\OO_X)\to H^1(K_X)$ has rank $q-2$;
\item  if $\ker \cup s=<\bar \al,\bar \beta>$, with $\al,\beta\in H^0(\Omega^1_X)$, then  $\al\wedge \beta\ne 0$.
\end{enumerate}
Then the hypersurface $Z$ defined by  $\Pf$ (cf. Lemma \ref{lem:qeven-odd}) is smooth at $[s]$.
\end{lem}
\begin{proof}  
Set $w=\alpha\wedge \beta$; 
we are going to show that the line  $M\subset |K_X|$ that joins $[s]$ and $[w]$ intersects  $Z$  with multiplicity 1 at $[s]$. 
We  regard $c_s$ and $c_w$ as alternating forms on $H^1(\OO_X)$. Let $J:=\left(\begin{array}{cc} 0&1\\-1&0\end{array}\right)$;  it is possible to complete $\alpha, \beta$ to a basis of $H^1(\OO_X)$ in such a way that   the matrices $A$, $B$ associated to $c_s$, $c_w$  have the form:
$$A=\left(\begin{array}{ccccc} 0& 0 &\dots&\dots & 0\\
0& J & 0& \dots & 0\\
\dots&\dots &\dots &\dots &\dots\\
0&\dots& 0&J&0\\
0& \dots&\dots &\dots& J\end{array}\right)\qquad B=\left(\begin{array}{c|c}C&-\,^t\!N \\  \hline \\  N & M\end{array}\right),$$
where   $M$ is a $(q-2)\times (q-2)$ antisymmetric matrix and  $C$ is a $2\times 2$ antisymmetric matrix. Notice  that $\det C\ne 0$ by the condition that $\alpha\wedge \beta \ne 0$. 
 Set $D(\lambda, \mu):=\det(\lambda A+\mu B)$;
then it is easy to see that the coefficient in $D(\lambda,\mu)$ of the monomial $\lambda^{q-2}\mu^2$ is equal to $\det C\ne 0$. Hence $\mu=0$ is a double root of $D(\lambda,\mu )$. Since  $D(\lambda,\mu)$ is the square of the Pfaffian $\Pf(\lambda,\mu)$ of $\lambda A+\mu B$, it follows that $\mu=0$ is a simple root  of $\Pf(\lambda,\mu)$. So the intersection multiplicity of $M$ and $Z$ at $[s]$ is equal to 1, and thus  $Z$ is smooth at  $[s]$.
\end{proof}
\begin{proof}[Conclusion of  the proof of Thm. \ref{thm:para2}]
(i) By Lemma \ref{lem:qeven-odd}, we have $\Sigma_{main}=|K_X|$, hence $|K_X|\subset \cP_{main}$ by Lemma \ref{lem:key}.
\medskip

(ii) 
If $q$ is even, then we have $\Sigma=\Sigma_{main}$ by  Lemma \ref{lem:qeven-odd}, and therefore $|K_X|\cap\cP_{main}=|K_X|\cap \overline{( \cP\setminus |K_X|)}=\Sigma$ by Lemma \ref{lem:key}.  By Lemma \ref{lem:qeven-odd}, the closed set  $\Sigma$ is irreducible and it is the zero locus of a polynomial $\Pf$ of degree $q/2$. In particular, if $q=2$ then $\Sigma$ is a hyperplane, so  we may assume  from now on that $q\ge 4$.  Counting dimensions, one sees that if $[s]\in \Sigma$ is general, then $\pr_2\inv([s])$ has dimension 1. This is equivalent to the fact that the  linear map  $\cup s$ has rank $q-2$. In addition, the kernel of $\cup s$ contains a general $\bar \alpha\in H^1(\OO_X)$, hence $s$ 
satisfies the assumptions of Lemma \ref{lem:sing}, and so $[s]$ is a smooth point of the hypersurface $Z$ defined by  $\Pf$. Thus $\Sigma=Z$ is a reduced and irreducible hypersurface of degree $q/2$. 
Finally, if $X$ has no irrational pencil of genus $>1$, then every $s\in H^0(K_X)$ such that $\cup s$ has rank $q-2$ satisfies the assumptions of Lemma \ref{lem:sing}, hence it corresponds to a smooth point of $\Sigma$. Conversely, an argument similar to the proof of Lemma \ref{lem:sing} shows that if $\cup s$ has rank $<q-2$, then $[s]\in \Sigma$ is singular.
\end{proof}
\section{Examples and open questions}\label{sec:examples}

The first two examples here show that the assumptions that $q(X)\ge \dim X+1$ and that $X$ has no fibration of Albanese general type cannot be removed from Theorem \ref{thm:para3}; the third one shows that when $q(X)=\dim X+1$ the canonical system $|K_X|$ may or may not be  exorbitant.
\begin{ex} \label{ex:q=n}
\underline{Varieties with $q(X)=\dim X$ and $p_g(X)<\chi(K_X)+q-1$.}\\
Our starting point is the  example, constructed in \cite[\S 4]{chen-hacon},  of a threefold $Y$ of general type and maximal Albanese dimension with $p_g(Y)=1$, $q=3$ and $\chi(K_Y)=0$, so that $p_g(Y)-(\chi(K_Y))+q(Y)-1)=-1<0$. 

Now let $H\in \Pic(Y)$ be very ample, let $D\in |2H|$ be a smooth divisor and let $p\colon X\to  Y$ be the double cover given by the relation $2H\equiv D$.
The variety $X$ is a smooth threefold of general type and of maximal Albanese dimension. The usual formulae for double covers give:
\begin{gather*}
\chi(K_X)=\chi(K_Y)+\chi(K_Y+H)=\chi(K_Y+H),\\  q(X)=q(Y)+h^1(-H)=3+h^1(-H),\\
p_g(X)=p_g(Y)+h^0(K_Y+H)=1+h^0(K_Y+H)
\end{gather*}
Since $h^1(-H)=0$ and $\chi(K_Y+H)=h^0(K_Y+H)$ by Kodaira vanishing, we get $q(X)=3$ and $p_g(X)-(\chi(K_X))+q(X)-1)=-1$. By taking $H$ a multiple of a fixed very ample $L\in \Pic(Y)$, one obtains  examples with $\chi(K_X)$ arbitrarily large.
\medskip

In addition, in \cite{chen-hacon} $Y$ is constructed as a desingularization of a certain $\Z_2^2$-cover of a product of three elliptic curves and it is easy to see that replacing the elliptic curves by abelian varieties one can obtain examples of varieties $Y$ of any dimension $\ge 3$ that are of general type and maximal Albanese dimension and have  $p_g(Y)=1$, $\chi(K_Y)=0$  and $q(Y)=\dim Y$.  Taking a double cover $X\to Y$ as above,  one obtains examples with $\chi(K_X)$ arbitrarily large and $p_g(X)-(\chi(K_X)+q(X)-1)=2-\dim X<0$. 
\end{ex} 
\begin{ex}\underline{Varieties with Albanese general type fibrations.}\\
As in Example \ref{ex:q=n}, let $Y$ be a variety of general type and maximal Albanese dimension with $q(Y)=\dim Y=:n\ge 3$, $p_g(Y)=1$ and $\chi(K_Y)=0$. Let $C$ be a curve of genus 2 and let $Z:=Y\times C$; the variety $Z$ is of general type and of Albanese general type, with $\dim Z=n+1$, $q(Z)=n+2$, $p_g(Z)=2$, hence $p_g(Z)-(\chi(K_Z)+q(Z)-1)=1-n<0$. As in Example \ref{ex:q=n}, by taking a double cover $X\to Z$ branched on a smooth very ample divisor, one obtains examples with $p_g(X)-(\chi(K_X)+q(X)-1)=1-n<0$.

Notice that all these examples have an irrational pencil of genus 2 induced by the second projection $X\to C$. 

\end{ex}
\begin{ex} \label{ex:q=n+1}
\underline{Varieties with $q(X)=\dim X+1$.}\\ Let $A$ be an abelian variety of dimension  $q$ and let $X\subset A$ be a smooth ample divisor. Using the adjunction formula and Kodaira vanishing,  one sees immediately that $q(X)=q=\dim X+1$ and $p_g(X)=\chi(K_X)+q-1$, hence $|K_X|$ is not exorbitant by Theorem \ref{thm:para3}.  By \cite[\S 3]{theta}, the same is true also when  $X$ is a desingularization of an irreducible theta divisor in a principally polarized abelian variety. In addition, standard   computations show that by taking a smooth double cover $X\to Z$ with ample branch locus where  is  $Z$ a variety such that $p_g(Z)=\chi(K_Z)+q-1$ and $q(Z)=\dim Z +1$, then the Albanese map of $X$ factorizes through the double cover $X\to Z$ and $p_g(X)=\chi(K_X)+q-1$, $q(X)=\dim X+1$. Iterating this construction,  one gets examples with $p_g=\chi+q-1$ and Albanese map of arbitrarily high degree.
\medskip

Next we describe some examples of $n$-dimensional varieties with $q=n+1$ for which  the difference $p_g(X)-(\chi(K_X)+q-1)$ can be arbitrarily large.
Let $D$ be a smooth ample divisor in an abelian variety $A$ of dimension $q:=n+1$, $n\ge 3$,  let $Y\subset \pp^n$  be a smooth hypersurface of degree $d$ and let $X\subset D\times Y$ be an $n$-dimensional general complete intersection of  very ample divisors. Since the morphism $X\to D$ is generically finite, the variety $X$ has maximal Albanese dimension; in addition, by  the Lefschetz Theorem we have $h^i(\OO_X)= h^i(\OO_{D\times Y})$ for $i<n$. 
Using K\"unneth formula, we get: 
$$h^i(\OO_{D\times Y})=h^i(\OO_D)= {q\choose i}, \  \  i\le n-2;\quad h^{n-1}(\OO_{D\times Y})=h^{n-1}(\OO_D)+h^{n-1}(\OO_Y).$$
In particular, we have $q(X)=q=n+1$. 
It follows:
$$p_g(X)=\chi(K_X)-\sum_{i=1}^{n}(-1)^{i}h^{n-i}(\OO_X)=\chi(K_X)+ q-1+h^{n-1}(\OO_Y).$$
Hence, by taking $d>>0$ we can make $p_g(X)-(\chi(K_X)+q-1)=p_g(Y)$ arbitrarily large.
\end{ex}

We finish by posing a couple of questions:

\begin{qst}
The examples with $q(X)=\dim X +1$ and $p_g-(\chi+q-1)>0$  described in Example \ref{ex:q=n+1} have non birational Albanese map. Are  there  similar examples with birational Albanese map?
\end{qst} 
\begin{qst} Can one describe $|K_X|\cap \cP_{main}$ also  for varieties of dimension $>2$?  For instance, if one could show that $\cP_{main}\cap |K_X|$ is a non degenerate subvariety of  $|K_X|$ then Corollary \ref{cor:base2} would extend to higher dimension. 
\end{qst}

\bigskip

\begin{minipage}{13.0cm}
\parbox[t]{6.5cm}{Margarida Mendes Lopes\\
Departamento de  Matem\'atica\\
Instituto Superior T\'ecnico\\
Universidade T{\'e}cnica de Lisboa\\
Av.~Rovisco Pais\\
1049-001 Lisboa, PORTUGAL\\
mmlopes@math.ist.utl.pt
 } \hfill
\parbox[t]{5.5cm}{Rita Pardini\\
Dipartimento di Matematica\\
Universit\`a di Pisa\\
Largo B. Pontecorvo, 5\\
56127 Pisa, Italy\\
pardini@dm.unipi.it}

\vskip1.0truecm

\parbox[t]{5.5cm}{Gian Pietro Pirola\\
Dipartimento di Matematica\\
Universit\`a di Pavia\\
Via Ferrata, 1 \\
 27100 Pavia, Italy\\
\email{gianpietro.pirola@unipv.it}}
\end{minipage}

\end{document}